\documentclass[a4paper,11pt]{amsart}

 \usepackage{amssymb,amsmath,latexsym,amsthm}

\usepackage{graphicx}
\usepackage[utf8]{inputenc}

\usepackage[T1]{fontenc}

 \usepackage{url} 
\usepackage[colorlinks=true,urlcolor=blue,linkcolor=blue,citecolor=red]{hyperref}
 \usepackage{color}

 \usepackage{url} 
 \usepackage{hyperref}

\def\C{\mathbb{C}}
\def\R{\mathbb{R}}

\def\Z{\mathbb{Z}}

\def\SO{{\sf{SO}}}

\def\Iso{{\sf{Iso}} }

\def\\R{\mathcal R}
\def\R_c{{\mathcal{R}_{conf}}}
\def\Iso{{\mathcal{R}_{isom}}}

\newtheorem{theorem}{{Theorem}}[section]
\newtheorem{proposition}[theorem]{{Proposition}}
\newtheorem{isom.ext}[theorem]{{Trivial isometric extension}}
\newtheorem{definition}[theorem]{{Definition}}
\newtheorem{corollary}[theorem]{{Corollary}}
\newtheorem{example}[theorem]{{Example}}

\newtheorem{notation}[theorem]{{ Notation}}

\title{The $1$-parametric $h$-principle for smooth conformal immersions of surfaces}
\author{Alaa Boukholkhal}

\begin{document}
\begin{center}
\begin{abstract} 
We reformulate the problem of finding conformal immersions of closed Riemannian surfaces in the language of the $h$-principle and we prove that the inclusion from the space of smooth conformal immersions to the space of immersions induces a bijection on the sets of path connected components. 
\end{abstract}
\end{center}
\maketitle

\section{Introduction}
It is known since the time of Gauss, that conformal structures on an orientable surface are in one to one correspondence with complex structures (or Riemann surface structures) on it. It is therefore, natural to ask whether any Riemann surface admits an embedding (or an immersion) in the Euclidean space $\mathbb{E}^{3}$ or in any Riemannian manifold in general. Since the work of Nash \cite{nash1954c1} on isometric embeddings, an affirmative answer to this question was given by Garsia \cite{garsia1960imbedding}, R\"uedy \cite{zbMATH03355435} and Ko \cite{ko1989embedding} (for orientable Riemannian ambients). 

The theorem of Nash-Kuiper \cite{nash1954c1} \cite{kuiper1955c1}, is one of the main inspirations of Gromov's $h$-principle theory. The $h$-principle can be defined as studying differential relations (equations and inequalities in the partial derivatives) via homotopy methods. Another famous example of an $h$-principle is the work of Smale-Hirsh \cite{smale}, \cite{hirsch} on regular homotopy of immersions that describes the topology of the space of immersions. 

In the following, we address the classification of smooth conformal immersions of surfaces in the $h$-principle framework, recall that
\begin{definition}
Let $(\Sigma,g)$ and $(M,h)$ be two Riemannian manifolds. An immersion $f:\Sigma \rightarrow M$ is conformal if the induced metric $f^*h=\lambda g$ for some positive function $\lambda$ on $\Sigma$.
\end{definition}

We will establish a $1$-parametric $h$-principle for the relation of smooth conformal immersions of surfaces. In other words, we classify the  path connected components of the space of such immersions. More precisely, we will show that:

\begin{theorem}\label{T}
Let $(\Sigma,g)$ be a closed orientable Riemannian surface, and let $(M,h)$ be a Riemannian manifold of dimension $\geq 3$. The inclusion from the space of smooth conformal immersions of $(\Sigma,g)$ in $M$, to the space of immersions of $\Sigma$ in $M$ induces a bijection on the set of path-connected components. 

\end{theorem}

Combining the classification of regular immersions in $\mathbb{R}^3$ in the Smale-Hirsh theorem \cite{smale}, \cite{hirsch} (see \cite{JAMES1966351} for more details), with theorem \ref{T}, we get: 

\begin{corollary}
Let $(\Sigma,g)$ be a closed orientable Riemannian surface with Euler characteristic $\chi$. Then, the space of smooth conformal immersions of $(\Sigma,g)$ in $\mathbb{E}^3$ has precisely $2^{2-\chi}$ path connected component.
\end{corollary}

Except from tori, there is still no explicit construction for conformal models of higher genus surfaces in $\mathbb{R}^3$, a recent
functional approach was proposed by Pinkall and al \cite{Pinkalfindingconf} to construct models in each connected component.

\section{The $h$-principle: an overview}
Let $M$ and $N$ be two manifolds and consider the $1$-jet space
$$J^{1}(M,N)=\{(x,y,L)\; | \; x\in M, y\in N, L\in\mathcal{L}(T_x M, T_y N)\}$$
where $\mathcal{L}(T_x M, T_y N)$ is the space of linear maps from $T_x M$ to $T_y N$. We have a natural projection $\pi :J^{1}(M,N) \longrightarrow M$. A subset $\mathcal{R}\subset J^1(M,N)$ is called a \textbf{differential relation}.

\begin{example} Let $g$ (respectively $h$) be a Riemannian metric on $M$ (respectively $N$), we define the relation of isometric immersions by
    $$\mathcal{R}_{isom}(g)=\{ (x,y,L)\in J^{1}(M,N); \:\;\; L^*h=g\} $$
\end{example}
    Let $\mathcal{R}$ be a differential relation. A section $\sigma : x\in M \rightarrow (x,F(x),L(x)) \in \mathcal{R}\subset J^{1}(M,N)$ of the fiber bundle $\pi :J^{1}(M,N) \longrightarrow M$ is called a formal solution, we denote by $bs(\sigma)$ the map $F$ and we call it the \textbf{base map} of the $1$-jet $\sigma$. If in addition, we have  $$ D_x F = L(x) , \;\;\;\; \; \forall x\in M$$
we call $\sigma$ a holonomic solution of $\mathcal{R}$. 

We denote by $\Gamma(\\R)$ (respectively $Sol(\\R)$) the space of formal (respectively holonomic) solutions of $\\R$. 
\begin{definition}[Gromov \cite{gromov1986}]
    We say that a differential relation $\mathcal{R}$  satisfies 
    \begin{itemize}
        \item \textbf{The $h$-principle}, if for any formal solution $\sigma$, there exists a homotopy $\sigma_t $ of formal solutions, such that $\sigma_0=\sigma$ and $\sigma_1$ is a holonomic solution.
        \item \textbf{The $C^{0}$-dense $h$-principle}, if for any formal solution $\sigma$, and any neighborhood $U$ of $bs (\sigma)$, there exists a homotopy $\sigma_t $ of formal solutions, such that $\sigma_0=\sigma$, $\sigma_1$ is a holonomic solution and $bs(\sigma_t) \in U$ for all $t$.
         \item \textbf{The $1$-parametric $h$-principle}, if the inclusion $\iota : Sol(\\R) \rightarrow \Gamma(\\R)$ induces a bijection $\pi_0(Sol(\\R)) \rightarrow \pi_0(\Gamma(\\R))$, where $\pi_0(.)$ denotes the set of path-connected components. .
        \item \textbf{The full $h$-principle}, if the inclusion $\iota : Sol(\\R) \rightarrow \Gamma(\\R)$ is a weak homotopy equivalence.
    \end{itemize} 
\end{definition}

We introduce now a type of differential relations that satisfy the $h$-principle.  
\begin{definition}
    Let $p=(x,y,L)\in \\R$ and let $H\subset T_xM$ be a hyperplane, we define the slice of $\\R$ for $H$ over $p$ by:
    $$\\R_{H,p}=\{ \widetilde{L}\in \mathcal{L}(T_xM,T_yN), \;\; \widetilde{L}|_H=L|_H \; \text{and} \; (x,y,\widetilde{L})\in \\R \}$$
\end{definition}
\begin{definition}
    A differential relation $\\R$ is $ample$ if for each slice, either it is empty or the convex hull of each path-component of the slice is the entire fiber, i.e. for any  $p=(x,y,L)\in \\R$ and any hyperplane $H\subset T_xM$, we have 
        $$IntConv(\\R_{H,p})=\{\widetilde{L}\in \mathcal{L}(T_xM,T_yN), \;\; \widetilde{L}|_H=L|_H \}$$
        where by $IntConv(\\R_{H,p})$ we mean the interior of the convex hull of the path component of $\\R_{H,p}$ containing $L$.
    
\end{definition}

\begin{theorem}[Gromov \cite{gromov1986}]
    Let $\\R\subset J^{1}(M,N)$ be an open and ample differential relation. Then, $\\R$ satisfies the full $h$-principle.
\end{theorem}

We now apply this machinery to recover the theorem of Smale-Hirsh \cite{smale}. We define the relation of immersions as follows:    $$\mathcal{R}_{imm}=\{ (x,y,L)\in J^{1}(M,N); \:\;\; L \text{ has maximal rank }\} $$
For $p=(x,y,L)\in \mathcal{R}_{imm}$, $H\subset T_xM$ a hyperplane and $v\in T_xM\setminus H$, the slice of $\mathcal{R}_{imm}$ for $H$ is given by: 
\begin{align*}
    \mathcal{R}_{imm}(p,H)=& \{ \widetilde{L}\in \mathcal{L}(T_xM,T_yN), \;\; \widetilde{L}|_H=L|_H \; \text{and} \; \widetilde{L} \text{ has maximal rank }\}\\
    =& \{ \widetilde{L}\in \mathcal{L}(T_xM,T_yN), \;\; \widetilde{L}|_H=L|_H \; \text{and} \; \widetilde{L}(v)\notin L(H)\}
 \end{align*}
For $dim M< dim N$, it not hard to see that $ \mathcal{R}_{imm}$ is ample. Moreover, $ \mathcal{R}_{imm}$ is clearly open. Therefore, we get: 
\begin{corollary}[Gromov \cite{gromov1986}] \label{h}
    Let $M^m$ and $N^n$ be two manifolds such that $m<n$. Then, the relation of immersions $\\R_{imm}$ satisfies the full $h$-principle.
\end{corollary}

When $M$ is a closed orientable surface of genus $d$ and $N=\mathbb{R}^3$. Applying the corollary above, we get 
$$\pi_0(Imm(M,\mathbb{R}^3))= \pi_0(C^0(M,\SO(3)))$$
With a little more work, we get that the space of smooth immersions of $M$ in $\mathbb{R}^3$ has precisely $2^{2d}$ path connected component (see \cite{JAMES1966351}, \cite{PinkallsurfaImm} for a more details).

\section{Conformal immersions relation and Teichm\"uller space}
Fix now $(\Sigma, g)$ to be an orientable closed Riemannian surface, and let $(M,h)$ be a Riemannian manifold of dimension $\geq 3$.

We define the relation of conformal immersions for a metric $g$ on $\Sigma$ by 
$$\mathcal{R}_{\text{conf}}(g) =\{ (x,y,L)\in J^1(\Sigma,M)\; | \; \exists \lambda_{x} >0 , \; L^{*}h = \lambda_{x} g \} $$

 Finding a conformal immersion $f : (\Sigma , g) \rightarrow (M,h)$ is equivalent to finding a holonomic solution for the relation $\mathcal{R}_{\text{conf}}(g)$. This relation contains the isometric immersion relation (for any metric conformal to $g$), which is known to satisfy the $h$-principle (see \cite{gromov1986}), in the $C^{1}$-setting. We will need more work 
to find smooth solutions.

\begin{notation}
     Denote by $Sol_{C^\infty}(\\R_{\text{conf}}(g))$ the set of holonomic solutions of $\\R_{\text{conf}}(g)$ whose base map is smooth; equivalently, the set of smooth conformal immersions.
 \end{notation}

The important particularity in dimension two, is the correspondence between conformal structures and complex structures, which in this setting gives rise to the Teichm\"uller space $Teich(\Sigma)$. We will focus on the case of positive genus surfaces since all metrics are conformal on the sphere by the uniformization theorem and hence the space of conformal immersions identifies with the space of immersions in this case.

Let $\Sigma$ be a closed orientable surface of genus $d\geq 1$ and denote by: 
\begin{itemize}
    \item $\mathcal{M}et(\Sigma)$ the set of Riemannian metrics on $\Sigma$.
    \item $\mathcal{D}iff$ the group of orientation preserving diffeomorphisms of $\Sigma$.
    \item  $\mathcal{D}iff_\circ$ the connected component of the identity in $\mathcal{D}iff$.
    \item $\mathcal{P}$ the group of smooth positive functions $\lambda: \Sigma \rightarrow \mathbb{R}_{>0}$ on $\Sigma$.
    
\end{itemize}
We consider the following actions 
\begin{itemize}
    \item  Pointwise conformally equivalent metrics : 
    $$ \mathcal{P} \times\mathcal{M}et(\Sigma) \rightarrow\mathcal{M}et(\Sigma)$$
    $$(\lambda, g) \mapsto \lambda . g $$
    \item Conformally equivalent metrics up to isotopy : 
    $$\mathcal{M}et(\Sigma)/\mathcal{P} \times \mathcal{D}iff_\circ\rightarrow\mathcal{M}et(\Sigma)/\mathcal{P}$$
    $$ ([g], \varphi ) \mapsto [\varphi^*g]$$
\end{itemize}

\begin{definition}
The quotient space $\frac{\mathcal{M}et(\Sigma)/\mathcal{P}}{\mathcal{D}iff_\circ} $ is called the Teichm\"uller space and we denote it by $Teich(\Sigma)$.

\end{definition}

It is worth noting that the flexibility of smooth conformal immersions of surfaces in theorem \ref{T} is due to the fact that the Teichm\"uller space is finite dimensional (of dim $2$ in the genus $1$ case and of dimension $6d-6$ for genus $d\geq 2$). In contrast, in higher dimensions the space of conformal structures is infinite-dimensional (see \cite{Ebinconformal}) which can explain the lack of flexibility in that setting (see \cite{Chern-Simons}). 

In addition, $Teich(\Sigma)$ can be locally parametrized by Riemannian metrics (see \cite{fischer1984purely} for genus $\geq 2$). Moreover, when $\Sigma$ has genus $1$, we can find a global parametrization by Riemannian metrics. More precisely, on $\C/\Z^{2}$, the family of metrics $g_w=|dz+\frac{i-w}{i+w} d\Bar{z}|^2$ with $w\in \mathbb{H}^{2}$ parameterize $Teich(\Sigma)$. 

For more details about Teichm\"uller theory, see \cite{hubbard2016teichmuller} or \cite{farb2011primer}. 

\section{Proof of the main theorem}

Let $(\Sigma,g_0)$ be an orientable Riemmanian surface and $(M,h)$ a Riemannian manifold. Our goal is to prove that the inclusion map 
$$Sol_{C^\infty}(\\R_{\text{conf}}(g_0)) \hookrightarrow Imm(\Sigma,M)$$ induces a bijection on the $\pi_0$ level, where $Sol_{C^\infty}(\\R_{\text{conf}}(g))$ is the space of smooth conformal immersions of $(\Sigma,g_0)$ in $(M,h)$ and $Imm(\Sigma,M)$ is the space of smooth immersions. 

We will first prove the surjectivity, which can be reformulated as:
\begin{proposition}\label{surj}
The relation of smooth conformal immersions of $(\Sigma,g_0)$ in $(M,h)$ satisfies the $C^0$-dense $h$-principle.

\end{proposition}

We will need two important ingredients: 
 
\begin{itemize}
    \item \textbf{The $\varepsilon$-isometric immersions relation:}  let $(\Sigma,g_0)$ and $(M,h)$ be as before, and consider for $\varepsilon >0$ the relation of $\varepsilon$-isometric immersions: 
$$\Iso(g_0,\varepsilon)=\{ (x,y,L)\in J^{1}(\Sigma,M); \:\;\; |L^*h-g_0| < \varepsilon \} $$ 
    The proof of Nash-Kuiper's isometric embedding theorem, in the language of differential relations, relies on proving the $h$-principle for the relation of $\varepsilon$-isometric immersions, for a well chosen family of metrics $(g_n)$ and parameters $(\varepsilon_n)$.
    \begin{definition}
        An immersion $f:(\Sigma,g_0)\rightarrow (M,h)$ is called short if $f^{*}h\leq g_0$.
    \end{definition}
    We have the following theorem \cite{nash1954c1}, \cite{kuiper1955c1}, \cite{gromov1986}: 
    \begin{theorem}\label{thm}

Let $\{f_p\}_{p\in B}$ be a continuous family of short immersions 
parametrized by a finite-dimensional ball $B$. 
Then for every $\varepsilon>0$, the family $\{f_p\}_{p\in B}$ can be deformed 
through immersions depending continuously on the parameter $p$ 
to a continuous family of $\varepsilon$-isometric immersions 
$\{\tilde f_p\}_{p\in B}$ which is arbitrarily $C^0$-close to the original family. 
Moreover, if the original family consists of smooth immersions, 
the homotopy may be chosen within the space of smooth immersions.
    \end{theorem}
For a proof, see \cite{spring} or \cite{Borrelli}.
    \item \textbf{Brouwer's fixed point theorem:} In order to find smooth solutions for the conformal immersions relation, we do not target the holonomic solutions directly. Instead, we construct a family of immersions and conclude using a reformulation of the Brouwer fixed point theorem: 
    
    \begin{proposition}\label{prop}
Let $\psi:\overline{B(p_0,r)} \rightarrow \mathbb{R}^n$ be a continuous map from the closed euclidean ball of center $p_0 \in \mathbb{R}^n$ and radius $r>0$ that satisfies 
$$ \parallel \psi(p)-p \parallel \leq r,$$
then $p_0 \in Im(\psi)$.
\end{proposition}
\begin{proof}
Applying Brouwer's fixed point theorem \cite{Brouwer} to the map $p\longmapsto p_0+p-\psi(p)$, implies that there exists $p_1\in\overline{B(p_0,r)}$ such that $p_0+p_1-\psi(p_1)=p_1$, and hence, $\psi(p_1)=p_0$.
\end{proof} 
\end{itemize}

We are now ready to prove proposition \ref{surj}:

\begin{proof}[Proof of proposition \ref{surj}]
Let $[g_0]\in Teich(\Sigma)$ be the projection in the Teichm\"uller space of the Riemannian metric $g_0$, and let $B_r$ be a closed ball centered at $[g_0]$ and of radius $r$ in $Teich(\Sigma)$. We choose a continuous family of Riemannian metrics that parametrize $B_r$ (for simplicity, we use the same notation $B_r$ and identify the two sets through the projection map). Choose $\varepsilon > 0 $ such that $2\varepsilon < r$. 

Now, start with a formal solution $\sigma \in \Gamma(\R_c)$, $\sigma(x)=(x,bs\, \sigma(x), L_\sigma(x))$ such that $$
L^{*}h=\lambda g_0$$ 
for some positive function $\lambda$.

Since $B_r$ and $\Sigma$ are both compact, we can deform $\sigma$  (there is a constant $c>0$ such that $bs\, \sigma$ is short for all the metrics $cg$, $g\in B_r$) continuously into a family 
$$g\in B_r \longmapsto \sigma^{g}\in \Gamma( \Iso(g,\varepsilon))$$
such that $bs\,\sigma^{g}$ is smooth for every $g$.

Using theorem \ref{thm}, there exists a continuous map 
$$(t,g)\in [0,1]\times B_r \longmapsto \sigma^{g}_t \in \Gamma( \Iso(g,\varepsilon))$$
such that: 
\begin{itemize}
    \item $\sigma^{g}=\sigma^{g}_{0}$ for every $g$.
    \item $\sigma^{g}_{1}\in Sol(\Iso(g,\varepsilon))$ for every $g$.
    \item $bs\,\sigma^{g}_{1}$ is smooth for every $g$.
\end{itemize}

Consider now the map 
$$\psi: g\in B_r \longmapsto \pi((bs\,\sigma^{g}_{1})^{*})\in Teich(\Sigma)$$
where $\pi:\mathcal{M}et(\Sigma)\rightarrow  Teich(\Sigma)$ is the projection map.

 The map $\psi$ is obviously continuous and $\varepsilon$ was chosen small enough so that $\psi$ satisfy $$ \parallel \psi(g)-g \parallel \leq r$$
Since $Teich(\Sigma)$ is homeomorphic to $\mathbb{R}^{6d-6}$ when $\Sigma$ has genus $d\geq 2$ and is homeomorphic to the upper half plane $\mathbb{H}^2$ in the genus one case, we conclude by proposition \ref{prop}, that there exists $\widetilde{g}\in B_r$, such that $\sigma^{\widetilde{g}}_{1}\in Sol(\R_c)$ with $bs\, \sigma^{\widetilde{g}}_{1}$ smooth. 

Finally, since we started with a family of smooth immersions, all the deformations are performed inside the space of smooth immersions $Imm(\Sigma,M)$.

\end{proof}

We will now give a proof of the injectivity part of theorem \ref{T} that can be stated as follows: 
\begin{proposition}\label{inj}
For $t\in [0,1]$, let $f_t:(\Sigma,g_0) \rightarrow (M,h)$ be a continuous curve of smooth immersions such that $f_0$ and $f_1$ are conformal. Then there exists a continuous curve of smooth conformal immersions $\widetilde{f}_t :(\Sigma,g_0) \rightarrow (M,h)$ such that $\widetilde{f}_0=f_0$ and $\widetilde{f}_1=f_1$.
\end{proposition}

Notice that, we can not apply the same proof as before, since there is no parametric version of Brouwer's fixed point theorem. More precisely, for a continuous curve of continuous (even smooth) maps from the closed ball to itself, one can not guarantee the existence of a continuous curve of fixed points.
\begin{example}

Let 
\[
\gamma: t \in (0,1] \mapsto \tfrac{1}{2}\sin\left(\tfrac{1}{t}\right)[-1/2, 1/2]
\]
and let \( D \) be the unit disk in \( \mathbb{R}^2 \). For each \( t \in (0,1] \), consider the unique isometry \( \varphi_t \) of the Poincaré disk with fixed point \( \gamma(t) \) and rotation number \( \tfrac{t\pi}{2} \). For \( t = 0 \), define $ \varphi_0 = \mathrm{id}_D $, the identity map on the disk.
The map
\[
\psi: (t,x) \in [0,1] \times D \longmapsto \varphi_t(x) \in D
\]
is continuous, and the set of fixed points is given by
\[
Fix_\psi = \{(t, \gamma(t)) \;|\; t > 0\} \cup \{(0, x) \;|\; x \in D\}.
\]
which is not path-connected, and therefore there exists no section \( [0,1] \to Fix_\psi \).
\end{example}

However, there is a weaker parametric version of Brouwer's fixed point theorem:  

\begin{theorem}[Browder's fixed point theorem \cite{Browder}]
Let $X\subset \mathbb{R}^n$ be a compact and convex subset, and let $ \psi:[0,1]\times X \rightarrow X $
 be a continuous map. Then, the set $\{(t,x)\in[0,1]\times X \;:\; \psi(t,x)=x\}$
 has a connected component which projects surjectively into $[0,1]$.
\end{theorem} 
Using this theorem, along with the fact that connected real analytic sets are path-connected, leads to proposition \ref{inj}. 
\begin{proof}[Proof of proposition \ref{inj}]
Starting from a curve of smooth immersions $f_t: \Sigma \rightarrow M$, where $f_0$ and $f_1$ are conformal for $g_0$, we can construct, as in the proof of proposition \ref{surj}, a homotopy in the space of immersions (see Theorem \ref{thm}) to a smooth family: 
$$H_0: (t,g) \in [0,1] \times B_r \mapsto \sigma^{g}_t \in Imm(\Sigma,M)$$
such that: 
\begin{itemize}
    \item $bs(\sigma^{g}_0)=f_0$ and $bs(\sigma^{g}_1)=f_1$ for every $g$.
    \item $\sigma^{g}_{1}\in Sol(\Iso(g,\varepsilon))$ for every $g$.
    \item $bs(\sigma^{g}_t)$ is smooth for every $g$ and $t$.
\end{itemize}

By Royden's analytic approximation theorem \cite{Royden}, there exists a homotopy $\{H_s\}_{s\in[0,1]}$ connecting $H_0$ to a map: 
$$ H_1: (t,g) \in [0,1] \times B_r \mapsto \widetilde{\sigma}^{g}_t \in Imm(\Sigma,M)$$
where: 
\begin{itemize}
    \item $bs(\widetilde{\sigma}^{g}_0)=f_0$ and $bs(\widetilde{\sigma}^g_1)=f_1$ for every $g$.
    \item $\widetilde{\sigma}^{g}_{1}\in Sol(\Iso(g,\varepsilon))$ for every $g$.
    \item $H_1$ is analytic with respect to $t$ and $g$ on $]0,1[\times B_{r}$.
\end{itemize}
We consider now the map: 
$$\psi: (t,g) \in [0,1] \times B_r \longmapsto \pi((bs(\widetilde{\sigma}^{g}_{1}))^{*}h)\in Teich(\Sigma) $$
Using Browder's theorem reformulated as in proposition \ref{prop}, the set $\psi^{-1}(g_0)$ has a connected component that projects surjectively into $[0,1]$. Moreover, since $\psi$
 is analytic on $]0,1[\times B_{r}$ and constant on $\{0\}\times B_r$ and on $\{1\}\times B_r$, the set $\psi^{-1}(g_0)$ is path connected \cite{analyticconnected}.

\end{proof}

\newpage
\bibliographystyle{plain}
\bibliography{Bib.bib}
\end{document}